\documentclass{article}

\usepackage{url}
\usepackage{color}

\usepackage{enumerate}
\usepackage[mathscr]{euscript}		
\usepackage{amsthm}

\usepackage{psfrag}			
\usepackage{hyperref}

\usepackage{times}
\usepackage{amsmath}
\usepackage{amssymb}

\usepackage{anysize}

\marginsize{2cm}{2cm}{1cm}{2.4cm}

\newcommand{\floor}[1]{{\lfloor #1 \rfloor}}
\newcommand{\Z}{\ensuremath{\mathbb{Z}}}
\newcommand{\N}{\ensuremath{\mathbb{N}}}
\newcommand{\R}{\ensuremath{\mathbb{R}}}

\newcommand{\E}{\ensuremath{\mathbb{E}}}
\renewcommand{\P}{\ensuremath{\mathbb{P}}}

\newtheorem{thm}{Theorem}[section]

\newtheorem{lem}[thm]{Lemma}
\newtheorem{conjecture}[thm]{Conjecture}

\theoremstyle{definition}

\theoremstyle{remark}
\newtheorem{remark}[thm]{Remark}

\numberwithin{equation}{section}

\date{}
\begin{document}
\baselineskip=14pt

\title{Ballisticity conditions for random walk in random environment}
\author{A. Drewitz $^{1,2,}$\thanks{
Partially supported by the International Research Training Group
``Stochastic Models of Complex Processes''.
}
\and A. F. Ram\'\i rez $^{1,}$\thanks{
Partially supported by Iniciativa Cient\'\i fica Milenio P-04-069-F.
} $^,$\thanks{
Partially supported by Fondo Nacional de Desarrollo Cient\'\i fico
y Tecnol\'ogico grant 1060738.}
}

\date{\today}
\maketitle
{\footnotesize
\thanks{$^1$ Facultad de Matem\'{a}ticas, Pontificia Universidad Cat\'{o}lica de Chile, Vicu\~{n}a Mackenna 4860, Macul, Santiago, Chile;
{\it e-mail:} \href{mailto:adrewitz@uc.cl}{adrewitz@uc.cl}, \href{mailto:aramirez@mat.puc.cl}{aramirez@mat.puc.cl}}

\thanks{$^2$ Institut f\"ur Mathematik, Technische Universit\"at Berlin, Sekr. MA 7-5, Str. des 17. Juni 136, 10623 Berlin, Germany;
{\it e-mail:} \href{mailto:drewitz@math.tu-berlin.de}{drewitz@math.tu-berlin.de}}
}

\begin{abstract}
 
 Consider a random walk in a uniformly elliptic i.i.d. 
random environment  in dimensions $d\ge 2$. In
  2002, Sznitman introduced for  each $\gamma\in (0,1)$ the ballisticity
  conditions $(T)_\gamma$ and $(T'),$ the latter being defined as the fulfilment of $(T)_\gamma$
  for all $\gamma\in (0,1)$. He proved that $(T')$ implies ballisticity and
  that for each $\gamma\in (0.5,1)$,  $(T)_\gamma$ is equivalent
  to $(T')$. It is conjectured that this equivalence holds for all $\gamma\in
  (0,1)$. Here we prove that for
 $\gamma\in (\gamma_d,1)$, where $\gamma_d$ is a dimension dependent constant
taking values in the interval $(0.366,0.388)$,  $(T)_\gamma$ is
equivalent to  $(T')$. This is achieved by a detour along the effective criterion, the fulfilment
of which we establish by a combination of  techniques developed by Sznitman
giving a control on the occurrence of atypical quenched exit distributions through boxes. \end{abstract}

{\footnotesize

}

\section{Introduction} We study the relationship between the ballisticity
conditions $(T')$ and $(T)_\gamma$, for $\gamma\in (0,1)$,
 introduced by Sznitman in \cite{Sz-02} for random walk in random environment (RWRE).  Given a site $x\in
\Z^d$, define the vector $\omega(x):=\{\omega (x,e):e\in  \Z^d, |e|=1\}$ with
$\omega (x,e)\in (0,1)$ and such that $\sum_{\vert e \vert = 1} \omega(x,e) = 1.$
We call the quantity $\omega:=\{\omega(x):x\in \Z^d\}$ an {\it
  environment}. Consider a Markov chain $\{X_n:n\ge 0\}$ on $\Z^d$ which jumps from
each site $x\in \Z^d$  to the nearest neighbour site $x+e$ with
probability $\omega(x,e)$. If the starting position of this chain is a site $x\in \Z^d,$
denote by $P_{x,\omega}$ its law on $(\Z^d)^{\N}$. Assume that the
environment $\omega$ is random and call $\mu$ its probability distribution.  The
{\it quenched law} of a RWRE is defined as the set of random probability measures $P_{x,\omega}$
with $x\in \Z^d$ under $\mu$. The {\it averaged} or {\it annealed law} of a RWRE is the set of
probability measures $P_x:=\int P_{x,\omega} \, d\mu$ with $x\in \Z^d$. 
 We will
suppose that $\mu$ is a product measure, i.e. the random variables
$\{\omega(x):x\in \Z^d\}$ are i.i.d. with respect to $\mu.$ We will furthermore assume that $\mu$ is 
{\it uniformly elliptic} which means that there exists a constant $\kappa>0$ such that

\begin{equation} \label{uniformEll}
\mu(\inf_{|e|=1} \omega(0,e)>\kappa)=1.
\end{equation}

Given a vector $l\in \mathbb S^{d-1},$ a RWRE is called {\it
  transient in the direction $l$} if $P_0$-a.s.

$$
\lim_{n\to\infty}X_n\cdot l=\infty.
$$
Moreover, it is called 
{\it
  ballistic in the direction $l$} if $P_0$-a.s.

$$
\liminf_{n\to\infty} \frac{X_n\cdot l}{n}>0.
$$
Using renewal techniques it is not difficult to prove that ballisticity in the
direction $l$ is equivalent to the law of large numbers
$\lim_{n\to\infty}\frac{X_n\cdot l}{n}=v$, with $v>0$ deterministic.
In dimension $d=1$ it is well known that transience does not necessarily imply
ballisticity. In dimensions $d\ge 2$  some fundamental questions about this
model remain open.

\medskip
\begin{conjecture}
 Transience in the direction $l$ implies ballisticity in the direction $l$.
\end{conjecture}
\medskip

Sznitman and Zerner \cite{SzZe-99} and Zerner \cite{Ze-02} proved that the
limit $\lim_{n\to\infty} X_n/n$ 
exists $P_0$-a.s. Subsequently, Sznitman
\cite{Sz-02} introduced the conditions $(T)$ and $(T')$ related to a
fixed direction $l\in \mathbb{S}^{d-1}$ which entail ballisticity. 
 Let $\gamma\in (0,1)$. We say that
condition {\it $(T)_\gamma$ relative to $l$ is satisfied} (written as $(T)_\gamma|l$) if for every $l'$ in
a neighborhood of $l$ one has that

$$
\limsup_{L\to\infty} L^{-\gamma}\log P_0(X_{T_{U_{l',b,L}}}\cdot l'<0)<0,
$$
for all $b>0$ and $U_{l',b,L}:=\{x\in \Z^d:-bL<x\cdot l'<L\}$ with $T_{U_{l',b,L}}$ denoting the first exit time
of $U_{l',b,L}.$
We say that
condition {\it $(T')$ is satisfied relative to $l$} (written as $(T')|l$)
 if condition $(T)_\gamma \vert l$ holds for every $\gamma\in (0,1)$.
We furthermore agree that condition $(T)$ {\it relative to $l$ is satisfied} and write
$(T) \vert l$ if $(T)_\gamma \vert l$ holds for $\gamma = 1.$
Let $\{e_j:1\le j\le d\}$ be the canonical generators of the additive
group $\mathbb Z^d$.
In dimension $d=1$, $(T)|e_1$  is equivalent to transience in the
direction $e_1$ (see Proposition 2.6 of \cite{Sz-01}) for
which one has nice criteria at hand. 
Using an alternative characterisation of $(T)_\gamma,$ in terms of transience
in a given direction,
one can in particular deduce that $(T)_\gamma|e_1$ is
equivalent to transience in the direction $e_1$
for any $\gamma > 0.$
 In \cite{Sz-02}, Sznitman  proved that any RWRE in
a uniformly elliptic environment which satisfies $(T')|l$, has a deterministic velocity
$$
v:=\lim_{n\to\infty} X_n/n, \quad P_0-a.s.,
$$
such that $v \cdot l > 0$, i.e. it is
ballistic. He also showed that  a central
limit theorem is satisfied, so that

$$
\frac{1}{\sqrt{n}}(X_{[n\cdot]}-[n\cdot]v)
$$
converges under $P_0$ in law on $D(\R_+,\R^d)$ to a Brownian motion with
non-degenerate covariance matrix.
 Furthermore, in \cite{Sz-02}, the following conjecture for
higher dimensions is stated:

\medskip
\begin{conjecture}
Let $d\ge 2$. For
each $\gamma\in (0,1)$ and $l\in\mathbb S^{d-1}$, $(T)_\gamma|l$  is
equivalent
 to  $(T')|l$.
\end{conjecture}
\medskip

\noindent Sznitman proved (see \cite{Sz-02}) that for
each $\gamma\in (0.5,1)$ and $l\in \mathbb S^{d-1}$,  $(T)_\gamma|l$ is equivalent to $(T')|l$.  The main result of this paper is the
following.

\medskip
\begin{thm} \label{TGammaCondThm}
Let $d\ge 2$ and

\begin{equation} 
\nonumber
\gamma_d:=\frac{\sqrt{3d^2 -  d} -d}{2d-1}.
\end{equation}
Then, for each $\gamma\in (\gamma_d,1)$ and $l\in\mathbb S^{d-1}$,
 $(T)_\gamma|l$ is equivalent to $(T')|l$.
\end{thm}
\medskip
\begin{remark} 
By direct inspection  one observes that
$\gamma_d$ is monotonically decreasing in $d.$ 
Therefore, $\gamma_\infty:=\lim_{d\to\infty}\gamma_d = \frac{\sqrt{3}-1}{2}$ exists and we obtain
$$
0.366\approx\gamma_\infty<\gamma_d\le \gamma_2\approx 0.387.
$$
\end{remark}

The proof of Theorem \ref{TGammaCondThm}  follows renormalization ideas introduced by
Sznitman \cite{Sz-01}, to control the probability of the existence of slowdown traps, and
passes through the so called {\it effective criterion} introduced by him in
\cite{Sz-02} as a tool which facilitates the checking of condition $(T')$.
The effective criterion in a given direction is a sort of high dimensional version of the well known
one-dimensional condition $\E(\rho)<1$ \cite{So-75}, which ensures ballisticity to the right
 of the RWRE and where $\rho=\omega(x,e_1)/\omega(x,-e_1)$. It introduces
 boxes $B:=\{x\in\mathbb Z^d: x\in R((-(L-2),L+2)\times (-\tilde{L},\tilde{L})^{d-1})\}$, where
 $R$ is a rotation that fixes the origin and such that $R(e_1)=l$, with $l$
 the direction in the definition of condition $(T')$.
 $\tilde{L}$ will usually be chosen large and has to satisfy $3 \sqrt{d} \leq \tilde{L} < L^3.$
 It is important then to
 obtain a good control on the decay for large $L$ of

\begin{equation}
\label{control}
\P(P_{0,\omega}(X_{T_B}\cdot l \geq L)\le e^{-L^\beta}),
\end{equation}
where $T_B$ is the first exit time from the box $B$ and where 
 $\beta\in (0,1)$  is an appropriately chosen parameter. Sznitman \cite{Sz-02}
proves the equivalence for $\gamma\in (0.5,1)$, between $(T)_\gamma|l$ and
$(T')|l$, establishing the equivalence between $(T)_\gamma|l$, the effective
criterion in the direction $l$ and $(T')|l$. To prove that $(T)_\gamma|l$
implies the effective criterion in the direction $l$, he bounds the quantity
(\ref{control}) for $\beta=\gamma$ through Chebychev's inequality.  An
ingredient in the proof of Theorem \ref{TGammaCondThm} of this paper, is the
use of a renormalization step which starts from a {\it seed estimate},
both introduced by Sznitman in \cite{Sz-01}, to obtain better controls of the
quantity (\ref{control}).  Nevertheless, the materialisation into something
useful via such an ingredient, requires a crucial step involving a
careful decomposition of the quantity analogous to $\E(\rho)$ (in the
one-dimensional case) entering the definition of the effective criterion.

In subsections \ref{effCritSubsec} and \ref{asDirSubsec},
we recall this criterion and introduce some notation which will
be needed afterwards, discussing the concept of asymptotic direction and
stating Lemma 2.3, which provides a non-trivial control for the quantity
(\ref{control}). The proof of Theorem \ref{TGammaCondThm} is the subject of
subsection \ref{proofMainResultSect}.
Section \ref{anomalousExitDistProofSect} is dedicated to proving Lemma 
\ref{anomalousExitDist}. For this purpose, in
subsection 3.1 we first recall a renormalization lemma of Sznitman
\cite{Sz-01}. In subsection 3.2, we prove the  seed estimate result, Lemma \ref{seedEstimateLemma},
 which is a modification of Lemma 3.3 of
 \cite{Sz-01}, under condition $(T)_\gamma$ instead of the stronger
 condition $(T).$ Then, in subsection 3.3, these estimates
are used to obtain a good control on (\ref{control}) proving Lemma 
\ref{anomalousExitDist}.

\section{Preliminaries and proof of Theorem \ref{TGammaCondThm}}
Here we prove Theorem \ref{TGammaCondThm}, showing that the
effective criterion is satisfied with respect to the so called asymptotic
direction. In subsection 2.1, we recall the definition of the effective
criterion and its equivalence to the fulfilment of  $(T')$ as well
as its equivalence to the fulfilment of
 $(T)_\gamma$ for some $\gamma\in (0.5,1)$ proved by Sznitman in
\cite{Sz-02}. In subsection 2.2, we recall that  $(T)_\gamma$ implies
the a.s. existence of a deterministic asymptotic direction for the walk.
Furthermore, we  state Lemma \ref{anomalousExitDist}, which gives a control on the quenched
exit probabilities from boxes appearing in the definition of the effective
criterion. The proof of this lemma is postponed to section 3. In subsection
2.3, departing from $(T)_\gamma$ for some $\gamma \in (\gamma_d, 0.5],$
we prove the effective criterion with respect to the asymptotic
direction. Finally, in subsection 2.4, we briefly explain how this implies
Theorem \ref{TGammaCondThm}.

\subsection{Equivalence between $(T')$ and the effective criterion} \label{effCritSubsec}
To prove Theorem \ref{TGammaCondThm} we will
employ the so-called effective criterion.
For positive numbers $L,$ $L'$ and $\tilde{L}$ as well as a space rotation $R$
around the origin we use the box specification ${\cal B}(R, L, L', \tilde{L})$
to describe the set of boxes of the form
$B:= \{x\in\mathbb Z^d:x\in R((-L,L') \times (-\tilde{L}, \tilde{L})^{d-1})\}.$ Furthermore, let
$$
\rho_{\cal B}(\omega) := \frac{P_{0,\omega} (X_{T_B} \notin \partial_+ B)}{P_{0,\omega} (X_{T_B} \in \partial_+ B)}
$$
where
\begin{equation} 
\nonumber
\label{bdPosPart}
\partial_+ B := \partial B \cap \{ x \in \Z^d: l\cdot x \geq L', \vert R(e_i) \cdot x \vert \leq \tilde{L}, i \in \{2, \dots, d\}\}.
\end{equation}
Here, $\partial B:=\{x\in \Z^d: d(x,B)=1\}$ with $d(x,B)$ the distance from
$x$ to $B$ in the
1-norm,
and for $U \subset \Z^d$ we denote by $T_U$ the first exit time $T_U := \inf\{n \in \N: X_n \notin U\}$ with
the convention $\inf \emptyset = \infty.$
We will sometimes write $\rho$ instead of $\rho_{\cal B}$ if the box we refer to is clear from the context. Note
that due to the uniform ellipticity assumption, $\P$-a.s. we have $\rho \in
(0,\infty)$. Given $l\in\mathbb{S}^{d-1}$, we say that the {\it effective
  criterion with respect to $l$} is satisfied if
$$
\inf_{{\cal B}, a} \big\{ c_{1}(d) (\log \frac{1}{\kappa})^{3(d-1)} \tilde{L}^{d-1} L^{3(d-1)+1} \E \rho_{\cal B}^a \big\} < 1.
$$	
Here, $c_1(d) > 1$ and $c_2(d) >1$ are dimension dependent constants and
$a$ runs over $[0,1]$ while ${\cal B}$ runs over the box-specifications
${\cal B} = (R, L-2, L+2, \tilde{L})$ with $R$ a rotation such that $R(e_1) = l,$
$L \geq c_{2}(d),$ $3\sqrt{d} \leq \tilde{L} < L^3$.

The equivalence between  $(T)_\gamma$ for any $\gamma\in (0.5,1)$
and  $(T')$ was established by Sznitman passing through the
effective criterion.
\medskip
\begin{thm}[Sznitman,  \cite{Sz-02}] \label{Sz24}
 For each $l \in \mathbb{S}^{d-1}$  the following
conditions are equivalent.
\begin{enumerate}
\item There is a $\gamma\in (0.5,1)$ such that
$(T)_\gamma|l$   is satisfied.
\item The effective criterion with respect  to $l$ is satisfied.

\item
$(T')|l$   is satisfied.

\end{enumerate}
\end{thm}
To prove Theorem \ref{TGammaCondThm} we will also take advantage of the effective
criterion with respect to a particular
direction $\hat v$, called the {\it asymptotic direction}.

\subsection{Asymptotic direction and atypical quenched exit distributions} \label{asDirSubsec}
Here we recall that under $(T)_\gamma|l$ the random walk has an asymptotic
direction $\hat v$. As it will be explained, this implies that it will be
enough to prove that $(T)_\gamma|l$ implies
the effective criterion with respect to $\hat v$.
In Corollary 1.5 of \cite{Sz-02}, $(T)_\gamma|l$ is shown to be
equivalent to the simultaneous fulfilment of the following conditions.
{\it \begin{enumerate}
\item[(i)] \begin{equation} 
\nonumber
\{X_n:n\ge 0\} \text{ is transient in the direction } l.
\end{equation}
\item[(ii)] For some $c >0,$ 
\begin{equation} \label{finiteExponExpect}
E_0 \exp\{ c \sup_{0 \leq n \leq \tau_1} \vert X_n \vert^\gamma\} < \infty.
\end{equation}
\end{enumerate}
}

\noindent  Here, $|\cdot|$ denotes the Euclidean norm and $\tau_1$ the regeneration time
defined as the first time $X_n \cdot l$ obtains a new maximum and never goes
below that maximum again,
i.e.

$$
\tau_1:=\inf\Big\{n\ge 1: \sup_{0\le k\le n-1}X_k\cdot l< X_n\cdot l\ {\rm and}\ 
\inf_{k\ge n}X_k\cdot l\ge X_n\cdot l\Big\}.
$$
Transience in the direction $l$ implies that $\tau_1$ is $P_0$-a.s. finite, see \cite{SzZe-99}.

Due to $(i)$, $(T)_\gamma|l$ implies that condition $(a)$ of Theorem 1 in \cite{Si-07} 
is fulfilled. Hence we have the existence of an asymptotic
direction $\hat{v} \in \mathbb{S}^{d-1},$ i.e. 
\begin{equation} \label{asDir}
\lim_{n \to \infty} X_n/\vert X_n \vert = \hat{v} \quad P_0-a.s.
\end{equation}
Furthermore, as it is explained in Theorem 1.1 of \cite{Sz-02}, under $(T)_\gamma|l$, $(T)_\gamma|l'$
holds if and only if $\hat v\cdot l'>0$. Therefore, if we establish that
$(T)_\gamma|l$ implies the effective criterion with respect to the asymptotic
direction $\hat v$, by the equivalence between parts $(b)$ and $(c)$ of Theorem \ref{Sz24},
we conclude that $(T')|\hat v$ holds, and hence $(T')|l$ also.

We now turn to two basic lemmas giving estimates on the occurrence of atypical
quenched exit distributions for the RWRE.
In the formulation of these results, $B$ denotes the box
$\{x\in\mathbb Z^d: x\in \hat{R}((-(L-2),L+2) \times (-3L,3L)^{d-1})\}$
where $\hat{R}$ is a rotation mapping $e_1$ to $\hat{v},$ cf. (\ref{asDir}).
A {\it typical} quenched exit distribution for the RWRE gives a large
probability to laws concentrated on the walk starting from $0$ exiting
 the box $B$ through the front part of the boundary
$\partial_+ B$ defined in (\ref{bdPosPart}). The first lemma, whose proof we omit, is a direct consequence
of Lemma 1.3 in \cite{Sz-02}.

\begin{lem}
\label{gamma-decay} Let $l\in\mathbb S^{d-1}$ and assume that $(T)_\gamma|l$
is satisfied.  Then
$$
-\delta_1:= \limsup_{L \to \infty} L^{-\gamma} \log P_0 (X_{T_B} \notin \partial_+ B) < 0.
$$
\end{lem}
The second lemma will turn out to be a key estimate in the proof of Theorem 
\ref{TGammaCondThm}.
For this purpose we define the function
\begin{equation} \label{fDef}
f(\beta) = d\Big(\beta - \frac{1}{1+\gamma}\Big) \frac{1+\gamma}{\gamma}
\end{equation}
for $\beta \in ((1+\gamma)^{-1},1).$
\begin{lem}
\label{anomalousExitDist} Assume that $(T)_\gamma| \hat{v}$
is satisfied.  Then, if $\beta\in ((1+\gamma)^{-1},1),$ for any $c > 0$ and $\zeta \in (0,f(\beta))$ we have
$$
-\delta_2 := \limsup_{L \to \infty} L^{-\zeta} \log \P \big(P_{0,\omega}
(X_{T_B} \in \partial_+ B)\le e^{-cL^{\beta}}\big) < 0.
$$
\end{lem}
Note that $\delta_2$ may well depend on $\beta,$ $c$ and $\zeta.$ However, we usually do not
name this dependence explicitly.
The proof of this lemma involves the use of renormalization ideas beginning
with a {\it seed estimate} lemma. We postpone it to section \ref{anomalousExitDistProofSect}.

\subsection{Proof of the effective criterion with respect to the asymptotic
  direction} \label{effCritAsDirSubsec}
As in the proof of Theorem \ref{Sz24} given by Sznitman in  \cite{Sz-02}, 
we will  show that the quantity $\E \rho^a_{\cal B}$
decays as a stretched exponential as $L \to \infty$ for a suitable choice of $a$ and ${\cal B}.$
 A key ingredient of the
proof turns out to be the use of the renormalization ideas of Sznitman to
obtain upper bounds for the probability of slowdown traps on the environment.   

We will only consider the case $\gamma\le 0.5$.
We set $a:= L^{-\alpha}$ for some $\alpha \in (0,1)$
and consider the boxes
$B:= \{x\in\mathbb Z^d: x\in \hat{R}((-(L-2),L+2) \times (-3L,3L)^{d-1})\}$
where $\hat{R}$ is a rotation mapping $e_1$ to $\hat{v},$ cf. (\ref{asDir}).
Uniform ellipticity yields $\P (P_{0, \omega}(X_{T_B} \in \partial_+ B) \leq \kappa^{c(d)L}) = 0$ for
some dimension dependent constant $c(d).$ Thus we can split $\E \rho^a_{\cal B}$ according to

\begin{equation} \label{expDecomposition}
\E \rho^a = (I)+(II)+(III),
\end{equation}
where
\begin{align*}
(I):=&\E\big(\rho^a, P_{0,\omega} (X_{T_B} \in \partial_+ B) > e^{-k_0 L^\gamma}\big),\\
(II):=&\E\big(\rho^a, e^{-k_1 L^{\beta_1}} < P_{0,\omega} (X_{T_B} \in \partial_+ B) \leq e^{-k_0 L^\gamma}\big)
\end{align*}
and
\begin{equation} \label{IIISum}
(III):=
\sum_{j=1}^n \E \big(\rho^a, e^{-k_{j+1}L^{\beta_{j+1}}} <
P_{0,\omega} (X_{T_B} \in \partial_+ B) \leq e^{-k_j L^{\beta_j}}\big).
\end{equation}
Here, $n$ and $k_0$ as well as $k_{j}$ and $\beta_{j}$ for $j \in \{ 1, \dots, n+1\}$ are positive constants
to be chosen later and
satisfying
\begin{equation} \label{betaAPrioriCond}
1 = \beta_{n+1} > \beta_n > \dots > \beta_1 > (1+\gamma)^{-1}
\end{equation}
as well as $k_{n+1}$ large enough. In fact, $k_1, \dots, k_n > 0$ can be chosen arbitrarily.
\begin{lem}
\label{I} For all $L>0,$
\begin{equation} 
\nonumber
(I) 
\leq e^{k_0L^{\gamma - \alpha} - \delta_1 L^{\gamma - \alpha}+o(L^{\gamma-\alpha})}.
\end{equation}
\end{lem}
\begin{proof}
By Jensen's inequality, we see that $(I)\le e^{ak_0 L^{\gamma}} P_0(X_{T_B} \notin \partial_+ B)^a$.
The conclusion now follows from Lemma \ref{gamma-decay}.
\end{proof}
\noindent From this lemma it follows that if we choose
\begin{equation} \label{alphaConstraint}
\alpha<\gamma
\end{equation}
and
\begin{equation} \label{kZeroConstraint}
k_0<\delta_1,
\end{equation}
there exist positive constants $c_1$ and $c_2$ such that for all $L>0,$

$$(I)\le c_1e^{-c_2L^{\gamma-\alpha}}.$$ 

\begin{lem}
\label{II}
 For all $L>0$,

\begin{equation} 
\nonumber
(II) \leq e^{k_1 L^{\beta_1 - \alpha} -\delta_1 L^\gamma+o(L^{\gamma})}.
\end{equation}
\end{lem}
\begin{proof} Note that
\begin{align*}
(II)&\leq e^{a k_1 L^{\beta_1}} \E (P_{0,\omega} (X_{T_B} \notin \partial_+ B)^a,
P_{0,\omega} (X_{T_B} \notin \partial_+ B) \geq 1-e^{-k_0 L^{\gamma}})\\
&\leq e^{k_1 L^{\beta_1 - \alpha}} P_0(X_{T_B} \notin \partial_+ B) (1-e^{-k_0 L^\gamma})^{-1},
\end{align*}
where to obtain the second line we used Chebychev's inequality.
\end{proof}
\noindent From this lemma we see that choosing 
\begin{equation} \label{betaIneq}
\beta_1 < 2 \gamma
\end{equation}
and $\alpha$ satisfying 
\begin{equation} \label{alphaSecondConstraint}
\alpha \in (\beta_1-\gamma, \gamma),
\end{equation}
one then has that
there exist positive constants $c_1$ and $c_2$ such that for all $L>0,$

$$
(II)\le c_1 e^{-c_2 L^{\gamma}}.
$$
Now, to control the third term, we  use the following lemma.

\begin{lem}
\label{III} Let the $\beta_j$'s be chosen as in (\ref{betaAPrioriCond}).
Then, for all $L>0,$ $j \in \{1, \dots, n\}$ and $\zeta \in (0, f(\beta_j)),$
 
\begin{equation}
\nonumber
\E\big(\rho^a, e^{-k_{j+1} L^{\beta_{j+1}}} \leq
P_{0,\omega} (X_{T_B} \in \partial_+ B) \leq e^{-k_j L^{\beta_j}}\big)
\le e^{k_{j+1} L^{\beta_{j+1}-\alpha}-\delta_2 L^{\zeta} + o(L^{\zeta})}
\end{equation}
where $f$ is defined as in (\ref{fDef}).
\end{lem}
\begin{proof}
We estimate
\begin{align*}
\E(\rho^a, &e^{-k_{j+1} L^{\beta_{j+1}}} \leq
P_{0,\omega} (X_{T_B} \in \partial_+ B) \leq e^{-k_j L^{\beta_j}})\\
 &\leq e^{k_{j+1}L^{\beta_{j+1}-\alpha}}
\P(P_{0,\omega} (X_{T_B} \in \partial_+ B) \leq e^{-k_j L^{\beta_j}}).
\end{align*}
Since $\beta_j > (1+\gamma)^{-1},$ the application of Lemma \ref{anomalousExitDist} yields the result.
\end{proof}

To prove the effective criterion with respect to $\hat v$,
it is enough to prove that the terms
 $(I),$ $(II)$ and $(III)$ of the decomposition (\ref{expDecomposition}) decay
stretched exponentially.
 As follows from the discussions subsequent to Lemmas \ref{I} and
\ref{II}, for $(I)$ and $(II)$ this is achieved by respecting (\ref{alphaConstraint}),
(\ref{kZeroConstraint}), (\ref{betaIneq}) and (\ref{alphaSecondConstraint}). It therefore remains to deal
with $(III).$ Since we may choose $\alpha < \gamma$ arbitrarily close to $\gamma,$ Lemma \ref{III}
assures the desired decay once the following set of inequalities is fulfilled:

\begin{align}
\left. \begin{array}{rl}
\frac{1}{1+\gamma}<\beta_1   &<2\gamma,\\
\frac{1}{1+\gamma} <\beta_2 &< \gamma+f(\beta_1),\\
\frac{1}{1+\gamma} <\beta_3  &<\gamma+ f(\beta_2),\\
&\vdots\\
\frac{1}{1+\gamma} <\beta_{n}  &<\gamma+ f(\beta_{n-1}),\\
 1 &<\gamma+ f(\beta_{n}).
\end{array} \right\} \label{betaAPosterioriCond}
\end{align}
Now define $F (x) :=\gamma+ f(x)$ and for $k\ge 1$,
 $F^{(k)} (x) := F\circ F^{(k-1)}(x)$ with $F^{(0)}(x)=x$. Then 
in particular (\ref{betaAPosterioriCond}) is fulfilled if
\begin{equation} \label{betaReformulation}
\frac{1}{1+\gamma}<\beta_j < F^{(j-1)}(\beta_{j-1}), \quad j \in \{1, \dots,n+1\},
\end{equation}
with $\beta_0:=2\gamma$.
Therefore, it is enough to choose $\gamma$ such that
$(1+\gamma)^{-1} <2\gamma$ and
$(F^{(j)}(2\gamma))_{j \ge 0}$ forms an increasing sequence with
\begin{equation} \label{limItCond}
1 < \lim_{j \to \infty} F^{(j)} (2\gamma).
\end{equation}
In this case we can choose the constants appearing in (\ref{IIISum}) according to
$n := \inf\{j \in \N : F^{(j)}(\gamma) > 1\}$ and $\beta_j$ as large as permitted
by (\ref{betaReformulation}).

Now in order to check (\ref{limItCond}) we 
solve the equation $x = F(x)$ for $x$ to obtain the (unstable) fixed-point
\begin{equation} 
\nonumber
x^* := \frac{d-\gamma^2}{(1+\gamma)d - \gamma}.
\end{equation}
Thus, we observe that it is sufficient to have
$2\gamma > x^*>(1+\gamma)^{-1}$ in order for (\ref{limItCond}) to be
fulfilled. For $\gamma \in (0, 0.5]$, it is easy to check that the second inequality
is satisfied. Furthermore, the first inequality
\begin{equation} 
\nonumber
2 \gamma > \frac{d-\gamma^2}{(1+\gamma)d - \gamma},
\end{equation}
is clearly true whenever
$$
 \gamma > \gamma_d=\frac{-2d + \sqrt{12d^2 - 4 d}}{2(2d-1)}.
$$

\subsection{Proof of Theorem \ref{TGammaCondThm}} \label{proofMainResultSect}
By Theorem \ref{Sz24} of Sznitman \cite{Sz-02}, it is enough to consider
the case in which $\gamma\in (\gamma_d,0.5]$.
Assume that $(T)_\gamma \vert l$ holds for some $l \in \mathbb{S}^{d-1}.$
It follows from Theorem 1.1 of \cite{Sz-02} 
 that $l \cdot \hat{v} > 0$ and that
$(T)_\gamma \vert l'$ is satisfied if and only if $\hat v \cdot l' > 0.$
In particular,
$(T)_\gamma \vert \hat{v}$ holds.
In the previous subsection we proved that if
$(T)_\gamma| \hat{v}$ is satisfied for some $\gamma\in (\gamma_d,0.5]$, then the effective criterion is satisfied with
respect to the asymptotic direction $\hat v$. Now, by the equivalence between
parts $(b)$ and $(c)$ of Theorem \ref{Sz24}, it follows that $(T')|{\hat v}$
is satisfied. 
Since $\hat v\cdot l>0$, it follows that $(T')\vert l$ is satisfied.

\section{Atypical quenched exit distribution estimates} \label{anomalousExitDistProofSect}
The aim of this section is to prove Lemma \ref{anomalousExitDist}. 
We will apply Lemma 3.2 of \cite{Sz-01}, which we recall in subsection 3.1 and a modification of Lemma 3.3 of
the same paper, which we prove in subsection 3.2. In subsection 3.3 we show
how these results imply Lemma \ref{anomalousExitDist}.

\subsection{Sznitman's renormalization lemma}

We introduce for $\beta, L>0$ and $w \in \Z^d$
the notation
$$
X_{\beta, L}(w) := -\log \inf_{x \in B_{1, \beta, L}(w)} P_{x,\omega} (X_{T_{B_{2,\beta, l}(w)}} \in \partial^* B_{2,\beta,L}(w)),
$$
where
$$
B_{1,\beta,L}(w) := \big\{x\in \Z^d: x\in
\hat{R} (w + [0,L] \times [0,L^\beta]^{d-1})\big\},
$$
$$
B_{2,\beta,L}(w) := \big\{x\in \Z^d: x\in \hat{R}(w+(-dL^\beta, L] \times (-dL^\beta, (d+1)L^\beta)^{d-1}) \big\}
$$
and
$$
\partial^* B_{2, \beta, L}(w) := \partial B_{2, \beta, L}(w) \cap B_{1, \beta, L}(w + Le_1).
$$
We now recall the statement of the renormalization result of \cite{Sz-01}.
\begin{lem}[Sznitman, \cite{Sz-01}] \label{renormalizationLemma}
Assume $d \geq 2$ and (\ref{uniformEll}).
Assume that $\beta_0 \in (0,1)$ and $f_0$ is a positive function defined on $[\beta_0,1)$
such that
\begin{equation} 
\nonumber
f_0(\beta) \geq f_0(\beta_0) + \beta - \beta_0 \quad \text{ for } \beta \in [\beta_0, 1)
\end{equation}
and, for $\beta \in [\beta_0,1), \zeta < f_0(\beta),$
\begin{equation*}
\lim_{\beta' \uparrow \beta} \limsup_{L \to \infty} L^{-\zeta} \log \P (X_{\beta_0,L}(0) \geq L^{\beta'}) < 0.
\end{equation*}
Denote by $f$ the linear interpolation on $[\beta_0,1]$ of the value $f_0 (\beta_0)$ at $\beta_0$ and the
value $d$ at $1.$ Then, for $\beta \in [\beta_0,1)$ and $\zeta < f(\beta),$
\begin{equation} \label{renormalizationEst}
\lim_{\beta' \uparrow \beta} \limsup_{L \to \infty} L^{-\zeta} \log P(X_{\beta, L}(0) \geq L^{\beta'}) < 0.
\end{equation}
\end{lem}

\subsection{Seed estimate under condition $(T)_\gamma$}
To prove Lemma \ref{anomalousExitDist} we will apply
 Lemma \ref{renormalizationLemma}. But we need to find an optimal function $f_0$
for which the assumption of this lemma are satisfied. That is the content of 
the so called seed estimate,  Lemma \ref{seedEstimateLemma}, which we will
prove in this subsection. This result is analogous to Lemma 3.3 of
\cite{Sz-01}, which assumes $(T')$ and in turn relies on Lemma 2.3 of the
same paper which gives a control for the annealed probability for the
fluctuations of the projection on the orthogonal complement of $\hat v$ of the
walk. Since we will instead only assume
condition $(T)_\gamma$, we need some control analogous to Lemma 2.3 of
\cite{Sz-01}. For completeness, we state such result, which was also proved by
Sznitman, as Theorem A.2 in \cite{Sz-02}. First we introduce as in \cite{Sz-02}, for
$z\in\Z^d$ the following notation for the orthogonal projection on the
orthogonal subspace of $\hat v$

$$
\pi(z):=z-z\cdot\hat v \hat v,
$$
and for $u\in\R$ and $l\in\mathbb S^{d-1}$,
 the last visit of $X_n$ to $\{x\in\Z^d:l\cdot x\le u\}$ is denoted by

$$
L_u^l:=\sup\{n\ge 0: X_n\cdot l\le u\}.
$$

\medskip

\begin{thm} [Sznitman, \cite{Sz-02}] 
\label{projectionControl} Assume that for some $\gamma\in (0,1]$, $(T)_\gamma$
holds with respect to $l\in\mathbb S^{d-1}.$ Then for any $c>0$,
$\rho\in (0.5,1]$,

$$
\limsup_{u\to\infty}u^{-(2\rho-1) \wedge \gamma\rho}\log P_0\Big(
\sup_{0\le n\le L_u^l} |\pi (X_n)|\ge c u^\rho\Big)<0.
$$

\end{thm}
\noindent Now, with the help of Theorem \ref{projectionControl}, we will prove the
following lemma, following closely the proof of Lemma 3.3 of \cite{Sz-01}. We
also include the whole proof in the paper for completeness.

\begin{lem} \label{seedEstimateLemma}
Let $\gamma\in (0,1)$ and assume that condition
 $(T)_\gamma|\hat v$ is satisfied. Then, for each
$\beta_0 \in (1/2,1)$, we have that for every $\rho>0$  and $\beta \in [\beta_0,1)$
\begin{equation} \label{seedEstimate}
\limsup_{L \to \infty} L^{-(\beta + \beta_0 - 1) \wedge \gamma \beta_0} \log
\P( X_{\beta_0, L} \geq \rho L^\beta) < 0.
\end{equation}
\end{lem}


\begin{proof}
We proceed as in the proof of Lemma 3.3 of \cite{Sz-01}, taking advantage of
Theorem \ref{projectionControl}.
Define $\chi := \beta_0 + 1 - \beta \in (\beta_0,1]$,
\begin{equation}
\nonumber
L_0 := \frac{L - \eta L^{\beta_0}}{\floor{L^{1-\chi}}}
\end{equation}
as well as
$$
\tilde{B}_1(w) :=\big\{x\in \Z^d:x\in
 \hat{R}([0,L_0] \times [0,L^{\beta_0}]^{d-1})\big\}
$$
and
$$
\tilde{B}_2(w) := \big\{x\in\Z^d:x\in\hat{R}((-dL^{\beta_0},L_0] \times (-\eta L^{\beta_0},(1+\eta)L^{\beta_0})^{d-1})\big\}
$$
for $w \in \Z^d$ and $\eta > 0.$
Keeping to the notation of \cite{Sz-01} we say that a point $w \in \Z^d$
is {\it bad} if
$$
\inf_{x\in\tilde{B}_1(w)}P_{x,\omega} (X_{T_{{\tilde{B}}_2 (w)}} \in \partial_+ \tilde{B}_2(w)) < 1/2
$$
 and {\it good} otherwise.
 Here
$\partial_+ \tilde{B}_2$ is defined as in (\ref{bdPosPart}).
By Chebyshev's inequality
\begin{equation} 
\label{firstWBadEst}
\P( w\text{ is bad})
\leq  2^{d+1} L_0 L^{(d-1)\beta_0}
\Big( P_0 \big (\sup_{0 \leq n \leq T_{L_0}^{\hat{v}}} \vert \pi (X_n) \vert \geq \eta L^{\beta_0} \big)
+ P_0(T^{-\hat{v}}_{dL^{\beta_0}} < \infty) \Big),
\end{equation}
where for  $v \in \R^d$ and $L \in \R$ we employed the stopping time
$
T_L^v := \inf \{ n \in \N : X_n \cdot v \geq L\}.
$
The first summand can now be estimated via Theorem  \ref{projectionControl}. This yields
\begin{equation}
\label{sideExitEst}
0  > \limsup_{L \to \infty} L^{-(2\beta_0 - \chi) \wedge \gamma \beta_0}
\log P_0 \big(\sup_{0 \leq n \leq T_{L_0}^{\hat{v}}} \vert \pi(X_n) \vert \geq \eta L^{\beta_0} \big).
\end{equation}
The second summand is estimated as in \cite{Sz-01} yielding due to
 (\ref{finiteExponExpect}) that

\begin{equation} \label{backwardExitEst}
\limsup_{L \to \infty} L^{-\gamma \beta_0} \log P_0 (T_{dL^{\beta_0}}^{-\hat{v}} < \infty) < 0
\end{equation}
Inserting the definition of $\chi,$ (\ref{sideExitEst}), 
(\ref{backwardExitEst}) and (\ref{firstWBadEst})  gives the estimate
\begin{equation} \label{wBadEst}
\limsup_{L \to \infty} L^{-(\beta + \beta_0 -1) \wedge \gamma \beta_0} \log P_0(w \text{ is bad}) < 0.
\end{equation}

We now consider a certain set of trajectories starting in $B_{1, \beta, L}(0)$ and leaving
$B_{2, \beta, L}(0)$ via $\partial^* B_{2, \beta, L}(0).$ We then show that if the points
$
j L_0 e_1,
$
$j \in \{0, \dots, \floor{L^{1-\chi}}\}$
are all good, the above set of trajectories has a probability larger than
$
e^{-\rho L^\beta}
$
to occur, hence it remains only to estimate the probability that one of the $jL_0 e_1$ is bad in an adequate
way.

Now to describe the above mentioned set of trajectories consider a walk starting 
in $B_{1, \beta_0, L}(0) \cap \tilde{B}_1 (j_0 L_0 e_1),$ some $j_0 \in \{0, \dots, \floor{L^{1-\chi}} -1\}$
and let it leave $\tilde{B}_2(j_0 L_0 e_1)$ through $\partial_+ \tilde{B}_2(j_0 L_0 e_1).$ From this point
of exit, the walk can reach $\tilde{B}_1((j_0+1)L_0 e_1)$ within $c(d) \eta L^{\beta_0}$ steps and
stay within $\tilde{B}_2((j_0 + 1) L_0 e_1)$ along this way. We then assume the walk to exit
$\tilde{B}_2((j_0 + 1) L_0 e_1)$ in the same way as $\tilde{B}_2(j_0 L_0 e_1),$ return to the box
$\tilde{B}_1 ((j_0+2) L_0 e_1)$ in the same way as before to $\tilde{B}_1 ((j_0+1) L_0 e_1)$ and so on.
When reaching $\partial_+ \tilde{B_2} (( \floor{L^{1-\chi}} -1) L_0 e_1),$ we want the walk to
enter $\tilde{B}_1 ( \floor{L^{1-\chi}} L_0 e_1)$ without leaving
$
B_{2, \beta_0, L} (0) \cap \tilde{B}_2 ( \floor{L^{1-\chi}} L, e_1)
$
and then exit $B_{2, \beta_0, L}(0)$ through $\partial^* B_{2, \beta_0, L}(0).$
These two requirements can be met within $2c(d) \kappa^{L^{\beta_0}}$ steps.

Now assume that the points $jL_0 e_1,$ $j \in \{0, \dots, \floor{L^{1-\chi}}-1\}$ are all good.
The strong Markov property applied to each of the exit and entrance times of the trajectories
described above, yields
\begin{align*}
P_{x,\omega} (X_{T_{B_{2, \beta_0, L}(0)}} \in \partial_+ B_{2, \beta_0, L}(0))
&\geq \Big( \frac12 \kappa^{c(d) \eta L^{\beta_0}} \Big)^{L^{1-\chi}} \kappa^{2c(d) \eta L^{\beta_0}}\\
& > \exp\{-\rho L^\beta\}
\end{align*}
for $\eta > 0$ small enough and all $x \in B_{1, \beta_0, L}(0).$
Thus, translation invariance of the environment yields
$
\P( X_{\beta_0, L} \geq \rho L^\beta) \leq L^{1-\chi} \P(0 \text{ is bad}),
$
which in combination with (\ref{wBadEst}) finishes the proof.
\end{proof}

\subsection{ Proof of Lemma \ref{anomalousExitDist}}
For $\varepsilon >0$ small enough we define
$f_{0, \varepsilon} : [(1+\gamma)^{-1} + \varepsilon, 1) \to [0,1]$ by

$$
f_{0,\varepsilon}(\beta):= \beta - (1+\gamma)^{-1}.
$$
Then, by Lemma \ref{seedEstimate}  the assumptions of Lemma \ref{renormalizationLemma}
are satisfied for $\beta_0 := (1+\gamma)^{-1} + \varepsilon$ and $f_0 := f_{0, \varepsilon}.$
Therefore, defining 
$f_\varepsilon: [(1+\gamma)^{-1} + \varepsilon, 1) \to [0,1]$ as
the linear interpolation between the value $\varepsilon$ at $(1+\gamma)^{-1} + \varepsilon$
and the value $d$ at $1,$ i.e.	
$$
f_\varepsilon(\beta):= d \Big( \beta - \frac{1}{1+\gamma} - \varepsilon \Big)
\frac{(1+\gamma)\left(1-\frac{\varepsilon}{d}\right)}{\gamma - \varepsilon - \gamma \varepsilon} + \varepsilon,
$$
by Lemma \ref{renormalizationLemma} we can see  that for $\beta \in [(1+\gamma)^{-1} + \varepsilon, 1)$
and $\zeta < f_\varepsilon (\beta),$ (\ref{renormalizationEst}) holds.

Using the strong Markov property applied at $T_{B_{2,\beta,L}(0)}$ we obtain for all $\beta \in (0,1)$
and $L > 0$ large enough
$$
P_{0,\omega} (X_{T_{B_{2, \beta, L}(0)}} \in \partial^* B_{2, \beta, L}(0) )
\leq P_{0,\omega} (X_{T_B} \in \partial_+ B) \cdot \kappa^{-2c(d)}
$$
and hence in combination with the previously established version of (\ref{renormalizationEst}),
for any $\beta \in [(1+\gamma)^{-1} + \varepsilon, 1)$
and $\zeta < f_\varepsilon (\beta),$
\begin{align*}
&\limsup_{L \to \infty} L^{-\zeta} \log \P(P_{0,\omega} (X_{T_B} \in \partial_+ B) \leq e^{-c L^{\beta}})\\
&\leq \limsup_{L \to \infty} L^{-\zeta} \log \P\left(X_{\beta,L}(0) \geq c L^{\beta} + 2c(d) \log \kappa \right)\\
&= -\delta_2 < 0.
\end{align*}
\noindent Letting $\varepsilon \downarrow 0,$
this proves Lemma \ref{anomalousExitDist}.

\medskip

\noindent {\bf Acknowledgement.} We  thank  A.-S. Sznitman for introducing this
topic to us and for useful discussions.


\begin{thebibliography}{Sim07}

\bibitem[Sim07]{Si-07}
Fran\c{c}ois Simenhaus.
\newblock Asymptotic direction for random walks in random environments.
\newblock {\em Ann. Inst. H. Poincar\'{e}}, 43(6):751--761, 2007.

\bibitem[Sol75]{So-75}
Fred Solomon.
\newblock Random walk in a random environment.
\newblock {\em Ann. Probab.}, 3:1--31, 1975.

\bibitem[SZ99]{SzZe-99}
Alain-Sol Sznitman and Martin Zerner.
\newblock A law of large numbers for random walks in random environment.
\newblock {\em Ann. Probab.}, 27(4):1851--1869, 1999.

\bibitem[Szn01]{Sz-01}
Alain-Sol Sznitman.
\newblock {On a class of transient random walks in random environment.}
\newblock {\em Ann. Probab.}, 29(2):724--765, 2001.

\bibitem[Szn02]{Sz-02}
Alain-Sol Sznitman.
\newblock An effective criterion for ballistic behavior of random walks in
  random environment.
\newblock {\em Probab. Theory Related Fields}, 122(4):509--544, 2002.

\bibitem[Zer02]{Ze-02}
Martin P.~W. Zerner.
\newblock A non-ballistic law of large numbers for random walks in i.i.d.\
  random environment.
\newblock {\em Electron. Comm. Probab.}, 7:191--197 (electronic), 2002.

\end{thebibliography}

\def\polhk#1{\setbox0=\hbox{#1}{\ooalign{\hidewidth
  \lower1.5ex\hbox{`}\hidewidth\crcr\unhbox0}}}

\end{document}